\documentclass{amsart}

\usepackage{hyperref}
\usepackage{amsrefs}
\usepackage{comment}
\usepackage{graphicx} 
\usepackage{float} 

\newcommand{\tri}{\mathbin{\overline{\nabla}}}

\newcommand{\etal}{\textit{et~al.}}
\newcommand{\m}{{\mathbf m}}
\newcommand{\n}{{\mathbf n}}
\renewcommand{\k}{{\mathbf k}}

\newcommand{\T}{{\mathbf t}}
\newcommand{\Sb}{{\mathbf S}}
\newcommand{\Hb}{{\mathbf H}}
\newcommand{\Ub}{{\mathbf U}}

\newcommand{\Gb}{{\mathbf G}}

\newcommand{\Sm}{\mathcal{S}}

\newcommand{\mC}{C_{\m}}
\newcommand{\nC}{C_{\n}}
\newcommand{\mD}{D_{\m}}
\newcommand{\mV}{V_{\m}}
\newcommand{\mE}{E_{\m}}

\newcommand{\BigComment}[1]{}

\providecommand{\U}[1]{\protect\rule{.1in}{.1in}}
\newtheorem{theorem}{Theorem}

\newcommand{\WR}{Wildberger \etal}

\newcommand{\p}{\!+\!}

\begin{document}
\title[{Hyper-Catalan and Geode Recurrences}]{Hyper-Catalan and Geode Recurrences  \\ and Three Conjectures of Wildberger}
\markright{Hyper-Catalan and Geode Recurrences}
\author{Dean Rubine}

\date{\today}

%


\begin{abstract}
The hyper-Catalan number  
$C[m_2,m_3,m_4,\ldots]$ counts the number of subdivisions of a roofed
polygon into $m_2$ triangles, $m_3$ quadrilaterals, $m_4$ pentagons,
etc.
Its closed form has been known since Erd\'elyi and Etherington, 1940.
In 2025,
\WR{} showed
its generating sum
$\Sb[t_2,t_3,t_4,\ldots]$ is a zero of the general geometric univariate polynomial.
We use that to derive a recurrence for hyper-Catalans, which expresses each 
in terms of other hyper-Catalans with smaller indices, generalizing the well-known Catalan convolution sum. 

Wildberger notes the factorization $\Sb-1=(t_2 + t_3 + t_4 + \ldots)\Gb$, where the factor $\Gb$ is called the Geode.
We derive a recurrence that let us express the Geode coefficients in terms of other hyper-Catalan and Geode coefficients,
and ultimately in terms of hyper-Catalans alone.
We use it to prove three conjectures of Wildberger, all closed forms for special cases of elements of $\Gb$.  
While the recurrence allows us to expand each Geode coefficient as an integer combination of hyper-Catalans, enabling calculation,
a closed-form for the general Geode coefficient remains unknown, as does what it counts.
\end{abstract}

\maketitle

\section{The geometric polynomial formula}
Wildberger and Rubine~\cite{Wildberger2025} call a univariate polynomial \textbf{geometric} when its constant is $1$ and its linear coefficient is $-1$.
It's straightforward to generalize the geometric polynomial formula into the general polynomial formula, but we won't need that here.

\begin{theorem}[The geometric polynomial formula, \WR{}, 2025~\cite{Wildberger2025}
] \label{thm:geompoly}
The generating sum for the hyper-Catalan numbers:
\begin{align}
\Sb \equiv \sum_{m_2,m_3,m_4, \dots \ge 0} C[m_2, m_3, m_4, \ldots] t_2^{m_2} t_3^{m_3} t_4^{m_4} \cdots \equiv \sum_{\m \ge 0} \mC \T^{\m}
\end{align}
is a formal series zero of the general univariate geometric polynomial or power series:
\begin{align} \label{eqn:galpha}
g(\alpha) = 1 - \alpha + t_2 \alpha^2 + t_3 \alpha^3 + t_4 \alpha^4 +{} \ldots
\end{align}
\end{theorem}

The $\sum_{\m \ge 0} \mC \T^{\m}$ notation expresses the zero as a sum over \textbf{natural vectors}, i.e. vectors of natural numbers. This allows us to deal with the explosion of terms in this power series of an unbounded number of variables about as easily as if there were a single integer index.

It's worthwhile to sketch their proof, which involves mapping the non-associative algebra of \textbf{subdigons}, subdivided roofed polygons, onto the usual algebra of polynomials.  
A subdigon $s$ that's divided into $m_2$ triangles, $m_3$ quadrilaterals, $m_4$ pentagons,
etc. is of \textbf{type} $\m=[m_2, m_3, m_4, \ldots]$.
Trailing zeros don't change the type, and $\m=[\,]$, the vector of all zeros, is the type of $|$, the \textbf{null subdigon}. 
$\Sm_{\m}$ is the 
multiset of subdigons of type $\m$, and its size $|\Sm_{\m}|$ is the \textbf{hyper-Catalan number} $\mC$, the number of subdigons of type $\m$. 
The \textbf{roof} of a subdigon is a distinguished side of the subdivided polygon.
The \textbf{central polygon} of a subdigon is the inner, un-subdivided polygon containing the roof side.  
$ s= \tri_k(s_1, s_2, \ldots, s_k)$ is the panelling operation that adjoins by its roof each subdigon $s_i$ consecutively counterclockwise to a central $k+1$-gon, whose remaining side is the roof of $s$ (Figure \ref{fig:tri4}).
Extending $\tri_k$ to multisets of subdigons,
\begin{equation}
\tri_k(M_1, M_2, \ldots, M_k) = [ \, \tri_k(s_1, s_2, \ldots, s_k)  : \, s_1 \in M_1, \, s_2 \in M_2, \ldots, s_k \in M_k \, ]
\end{equation} 
\begin{figure} 
\centering
\includegraphics[width=1.0\textwidth]{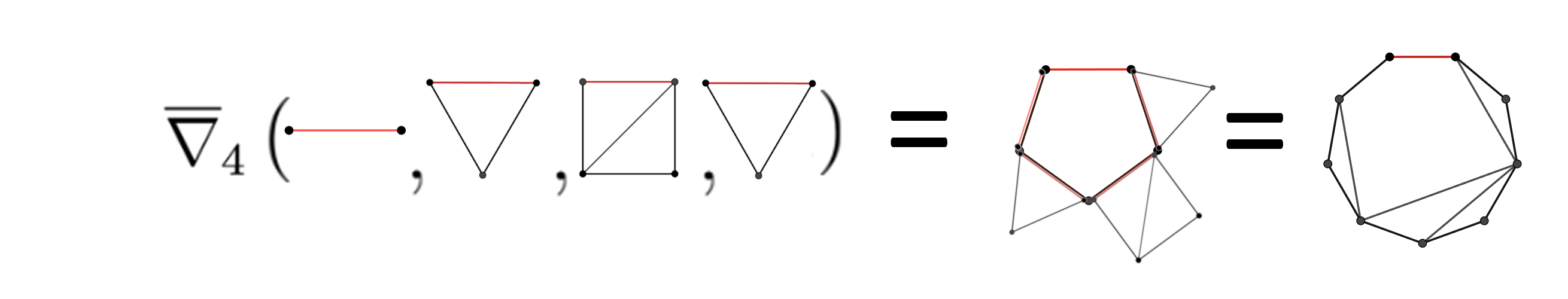}  %
\vspace{-20pt}
\caption{The algebra of subdigons: $\tri_4$  creates a subdigon with a central pentagon.}
\label{fig:tri4} 
\end{figure}
$\psi$ maps a subdigon $s$ of type $\m$ to its \textbf{accounting monomial},
\begin{equation}
 \psi(s)  \equiv t_2^{m_2} t_3^{m_3} t_4^{m_4} \cdots \equiv \T^{\m}
 \end{equation}
and $\Psi(M)=\sum_{s\in M}\psi(s)$ maps a multiset of subdigons to a polynomial that is the sum of the accounting monomials of the elements of $M$.
We have $\Psi(\Sm_{\m})=\mC \T^{\m}$.

We form the multiset of all subdigons $\Sm$ as $\Sm \equiv \sum_{\m \ge 0}\Sm_{\m} $
and define $\Sb \equiv \Psi(\Sm)$ so $\Sb \equiv \Sb[t_2, t_3, t_4, \ldots] = \sum_{\m \ge 0} \mC \T^{\m}$ is the generating sum in an unbounded number of variables for the hyper-Catalan array $\mC$.

The accounting of accounting monomials comes from the identity:
\begin{equation}
\psi(  \tri_k(s_1, s_2, \ldots, s_k) ) = t_k \, \psi(s_1) \psi(s_2)\cdots \psi(s_k) 
\end{equation}
because the multiplications are essentially additions of type vectors, whose components are exponents.
The result extends to multisets via the linearity of $\Psi$:
\begin{equation} \Psi(  \tri_k(M_1, M_2, \ldots, M_k) ) = t_k \, \Psi(M_1) \Psi(M_2)\cdots \Psi(M_k).\end{equation}

Each subdigon $s$ is either the null subdigon $|$ or has a central $k+1$-gon so is of the form
$ s= \tri_k(s_1, s_2, \ldots, s_k)$
for unique subdigons $s_1, s_2, \ldots, s_k$.
The recursive structure implies the multiset equation:
\begin{equation} \Sm = [ \ | \ ] + \tri_2(\Sm, \Sm) + \tri_3(\Sm, \Sm, \Sm) + \tri_4(\Sm, \Sm, \Sm, \Sm) + \ldots\end{equation}
Applying $\Psi$ to both sides, then applying the multiset identity:
\begin{align}
\Psi(\Sm) &= \psi(|) + \Psi( \tri_2(\Sm, \Sm))  + \Psi( \tri_3(\Sm, \Sm, \Sm) ) 
+ \ldots
\\
\Psi(\Sm) &= \psi(|) + t_2 \Psi( \Sm) \Psi(\Sm)  + t_3 \Psi(\Sm)\Psi(\Sm) \Psi(\Sm)   + t_4 (\Psi(\Sm))^4 + \ldots
\\
\Sb &= 1 + t_2 \Sb^2 + t_3 \Sb^3 +  t_4 \Sb^4 + \ldots
\end{align}
so $\Sb$ is the zero of the general geometric polynomial (equation \eqref{eqn:galpha}),
$g(\Sb) = 0$. \hfill \qedsymbol{}

\section{The hyper-Catalan Recurrence} 
Let's write the 
multinomial theorem to get an expression for $\alpha^j$ and ultimately $t_j \alpha^j$ so we can expand $g(\alpha)$ where $\alpha=\sum_{\n \ge 0} \nC \T^{\n}$, the generating series of $\nC$, which we'll initially consider to be an unspecified  array of unbounded dimension.
\begin{equation}
\alpha^j
=  \left( \sum_{\n \ge 0} \nC \T^{\n} \right)^j
= \sum_{\substack{ \sum_{\n \ge 0} k_{\n} =j }} \binom{j}{\bf{k}} \prod_{\n \ge 0}  ( {C}_{\n} \T^{\n} ) ^{k_{\n}}
\end{equation}
In $\alpha$ our terms are indexed by subdigon types instead of natural numbers. For each iteration of the right hand sum, each $\nC \T^{\n}$ term has an associated $k_{\n} \ge 0$. 
We abbreviate the bottom of the multinomial coefficient as 
$\k$; in it we need only include the finite number of non-zero $k_{\n}$s and their order does not matter.
\begin{equation}
\alpha^j
= \sum_{\substack{ \sum_{\n \ge 0} k_{\n} =j }} \binom{j}{\bf{k}} \prod_{\n \ge 0}C_{\n} ^{k_{\n}} \prod_{\n \ge 0}\T^{k_{\n} \n}  
\end{equation}
We're interested in the general term, $[\T^{\m}]\alpha^j$.
Those are the terms where
$\prod_{\n \ge 0} \T^{k_{\n} \n} =\T^{\m}$ or 
$\sum_{\n \ge 0} k_{\n} \n = \m $.
We end up summing over the vector partitions of $\m$ with $j$ parts~\cite{French2018}.
In other words, the term $C_{\n} \T^{\n}$ from $\alpha$ only contributes to $[\T^{\m}]\alpha^j$ when there are natural numbers $\k$ which sum to $j$ and which include a $k_{\n}>0$ such that the linear combination $\sum_{\n \ge 0} k_{\n} \n $ makes $\m$.
\begin{equation} \label{eqn:alphaj}
[\T^{\m}]\alpha^j
= \sum_{\substack{ \sum_{\n \ge 0} k_{\n} =j    \\ \sum_{\n \ge 0} k_{\n} \n = \m   }} \binom{j}{\k} 
\prod_{\n \ge 0} C_{\n} ^{k_{\n}} 
\end{equation}
At this point, this has nothing in particular to do with the hyper-Catalans. It's the $j$th power of a general multivariate series with $C$s as coefficients.
We seek $[\T^{\m}]t_j \alpha^j $.
Let's define basis vectors 
$\vec{j} \equiv [0,0,\ldots, 1]$ where there are $j-2$ zeros.
\begin{equation} 
[\T^{\m}] t_j \alpha^j 
= [\T^{\m - \vec{j} }]  \alpha^j
= \sum_{\substack{ \sum_{\n \ge 0} k_{\n} = j \\ \sum_{\n \ge 0} k_{\n} \n = \m-\vec{j} }} \binom{j}{\k} 
\prod_{\n \ge 0} C_{\n} ^{k_{\n}} 
\end{equation}
We implicitly have the constraint $m_j > 0 $ in the sum, for if $m_j=0$ then $\m-\vec{j}$ would have $-1$ for its $j^{\textrm{th}}$ component, so cannot be a linear combination of natural vectors $\n$ using weights $k_{\n} \ge 0$. When $m_j=0$, clearly $[\T^{\m}] t_j \alpha^j = 0$ because all the terms have a factor of $t_j$; thus $[\T^{\m - \vec{j} }]  \alpha^j = 0$ when $m_j=0$. 

Recall $g(\alpha)=1-\alpha + \sum_{j \ge 2} t_j \alpha^j$ (equation \eqref{eqn:galpha}). 
For $\m \ne [ \ ]$,
\begin{equation} 
[\T^{\m}] g(\alpha) 
= - [\T^{\m}] \alpha + \sum_{j \ge 2} [\T^{\m}] t_j \alpha^j 
= - [\T^{\m}] \alpha + 
\sum_{j \ge 2} \!
\sum_{\substack{ \sum_{\n \ge 0} k_{\n} = j \\ \sum_{\n \ge 0} k_{\n} \n = \m-\vec{j} }} \!\!\! \binom{j}{\k} \prod_{\n \ge 0} C_{\n} ^{k_{\n}} 
\end{equation}
That's true for $g(\alpha)$ applied to $\alpha=\sum_{\n \ge 0} \nC \T^{\n}$ for any array $\nC$. 
From here on we use $\nC$ to denote the hyper-Catalan array,
so $\alpha=\Sb$,  $[\T^{\m}] \alpha =C_{\m}$, and by Theorem \ref{thm:geompoly}, $g(\alpha)=0$.
We've derived the hyper-Catalan recurrence.
\begin{theorem}[Hyper-Catalan Recurrence] \label{thm:hcat}
$C_{[ \, ]}=1$ (that's the Catalan number $C_0$) and for $\m \ne [ \, ]$:
\begin{equation*}
C_{\m} = \sum_{j \ge 2} \ \ \sum_{\substack{ \sum_{\n \ge 0} k_{\n} = j \\ \sum_{\n \ge 0} k_{\n} \n = \m-\vec{j} }} \binom{j}{\k} \prod_{\n \ge 0}  C_{\n} ^{k_{\n}} 
\end{equation*}

\end{theorem}
As an example, let's try $C[1,1]$.
Non-zero $m_2$ and $m_3$ make the outer sum just over $j=2$ and $j=3$.  For the first we're aiming for $\m-\vec{j}=[0,1]$; for the second $\m-\vec{j}=[1,0]=[1]$.
The $[0,1]$ gives a single term with $j=2, k[\,]=1, C[\,]=1, k[0,1]=1, C[0,1]=1$ and the $[1]$ gives a single term with $j=3, k[\,]=2, C[\,]=1, k[1]=1, C[1]=1$.  So
\begin{equation}
C[1,1] =  \binom{2}{1,1} 1^1 1^1 + \binom{3}{2,1} 1^1 1^1 = 2+3=5  \quad\checkmark    
\end{equation}
That's correct; it may be checked using Theorem \ref{thm:hypercatalanclosedform} below.

The well-known Catalan convolution identity $C_{m+1} = \sum_{n=0}^{m}  C_{n} C_{m-n}$
may be derived directly from the hyper-Catalan recurrence.
We leave that for the reader.
In his Section 2.3.4.4, Knuth~\cite{Knuth1997} infers the Catalan convolution from a simple argument counting binary trees.

\section{Multiparameter Fuss-Catalan numbers}

Wildberger's polynomial formula uses $\mC$, the array of hyper-Catalan numbers. 

\begin{theorem}[Hyper-Catalan Closed Form, Erd\'elyi and Etherington, 1940~\cite{Erdelyi1940}] \label{thm:hypercatalanclosedform}
The number of subdigons of type $\m=[m_2, m_3, m_4, \ldots]$ is
\begin{align*}
\mC &=\dfrac{( 2m_2 + 3m_3 + 4m_4 + \ldots )!}{(1 + m_2 + 2m_3 + 3m_4 + \ldots)! \,  m_2! \, m_3! \,  m_4! \cdots} .
\end{align*}
\end{theorem}
\newcommand{\mMu}{\boldsymbol\mu}

That turns out to be a special case of what S. R. Mane \cite{Mane2024} calls the \textbf{multiparameter Fuss-Catalan numbers}:
\begin{equation}
A_{\m}( \mMu, r) = \dfrac{r}{\m !} \prod_{j=1}^{-1 +\sum_i m_i} (\m \cdot \mMu +r -j) .
\end{equation}
Raney \cite{Raney1960} showed (for arrays of integers $\mMu$):
\begin{equation}
\sum_{\m \ge 0} A_{\m}( \mMu, r) \  \T^{\m} =\left( \sum_{\n \ge 0} A_{\n}( \mMu, 1) \  \T^{\n} \right)^r  
\end{equation}
Graham \cite{Graham1989} generalizes the (singleton case) $\mMu$ to a complex number, which generalizes the polynomials we're solving into signomials, sums of terms where the exponents are arbitrary complex numbers.  We'll continue to focus on polynomials.

Taking $\m=[m_2, m_3, m_4, \ldots]$ and defining $\mC^{(r)} \equiv A_{\m}( [2,3,4,\ldots], r) $:
\begin{equation}
\mC^{(r)}  =\dfrac{r ( r-1 + 2m_2 + 3m_3 + 4m_4 + \ldots )!}{(r + m_2 + 2m_3 + 3m_4 + \ldots)! \,  m_2! \, m_3! \,  m_4! \cdots}  .
\end{equation}
\begin{equation}
\mC^{(r)}  =\dfrac{r ( r-2 + \mE  )!}{(r-2  +\mV)! \,  m_2! \, m_3! \,  m_4! \cdots}  .
\end{equation}
we have $\mC = \mC^{(1)}$ and per Raney:
\begin{equation} 
\mC^{(r)}  = [ \T^{\m} ] \left( \sum_{\n \ge 0} \nC \T^{\n} \right)^r   =  [ \T^{\m} ] \Sb ^r
\end{equation}

Combining with equation (\ref{eqn:alphaj}), we get:

\begin{theorem}[Powers of \bf S]
\begin{equation}
\mC^{(r)} 
 =\dfrac{r ( r-1 + 2m_2 + 3m_3 + \ldots )!}{(r + m_2 + 2m_3 +  \ldots)! \,  m_2! \, m_3! \cdots}  
= \sum_{\substack{ \sum_{\n \ge 0} k_{\n} =r    \\ \sum_{\n \ge 0} k_{\n} \n = \m   }} \binom{r}{\k} 
\prod_{\n \ge 0} C_{\n} ^{k_{\n}} 
\end{equation}
where $\Sb$ is the generating function for the hyper-Catalan numbers, $\mC = [ \T^{\m} ] \Sb$, and $\mC^{(r)} =  [ \T^{\m} ] \Sb ^r$.
\end{theorem}

For example, let's take $r=3$, $\m=[102]$ (we omit commas and sometimes brackets in type vectors for brevity) which means we want  $C^{(3)}_{102} = [t_2 t_4^2 ] \Sb^3$.  Cubing the relevant part of $\Sb$,
\begin{align}
      C^{(3)}_{102} &= [ t_2 t_4^2](1 + t_2+t_4 + 6t_2t_4+ 4t_4^2 + 45t_2 t_4^2 )^3 
\\ &  \nonumber 
= [ t_2 t_4^2]( \ldots 
 + 3(45t_2 t_4^2) + 6(t_2)(6t_2t_4) + 6(t_2)(4t_4^2) + 3(t_2)(t_4)^2 + \ldots 
)
\\ &  \nonumber  
= 198.
\end{align}

The closed form formula says:
\begin{equation} C^{(3)}_{102}  =\dfrac{3 ( 3-1 + 2(1) + 4(2))!}{(3 + 1+ 3(2))! \,  1! \, 2!}  =198. \quad\checkmark
\end{equation}
We see that to collect terms of type $\m$ on the 3rd power of our series, our sum is over vector partitions of $\m=[102]$ into $r=3$ parts; there are four:
\begin{align}
& \m=2[\,]+1[102], && {\textstyle \binom{3}{2,1} } C_{[\,]}^{2} C_{102}^1= (3)1^2 45^1 =135
\\
& \m=1[\,]+1[001]+1[101], &&   {\textstyle \binom{3}{1,1,1} }C_{[\,]}^{1} C_{001}^1 C_{101}^1= (6) 1^1  1^1   6^1 =36
\\
& \m=1[\,]+1[002]+1[1], &&   {\textstyle \binom{3}{1,1,1}} C_{[\,]}^{1} C_{002}^1 C_{1}^1= (6) 1^1 4^1 1^1 =24
\\
& \m=2[001]+1[1], &&   {\textstyle \binom{3}{2,1} } C_{001}^{2} C_{1}^1 = (3) 1^2 1^1=3
\end{align}
so  our recurrence says 
\begin{equation}
C^{(3)}_{102} =  [ t_2 t_4^2]\Sb ^3 = 135+36+24+3=198. \ \ \checkmark \ 
\end{equation}
For each vector $\n$ in the partition we get a factor of $\nC$ in the term; the multinomial coefficient counts how many times this term happens.

S. R. Mane, by collecting historical results and adding his own research, has made great advances in the series solution to polynomial equations. In particular, geometric form is just one way to make a pair of coefficients $\pm 1$, and Mane lays out how each such transformation gives rise to different series solutions, each with different convergence properties. In geometric form, $0=1-\alpha + \ldots$ we are essentially perturbing the root $\alpha=1$; if instead we normalize to $0=1 \pm \alpha^n + \ldots$ we are perturbing $n$th roots of unity or $-1$, so we get $n$ series roots for signomials with rational exponents.

\section{The Geode Factorization}
\WR{} uncover the Geode while examining a face layering of $\Sb$, which they call $ \Sb_F  = \Sb[ft_2, ft_3, ft_4, \ldots]$.
Expanding, they get a layering of $\Sb$ into ongoing power series
of a given total degree, accounting for subdigons with a given number of faces, and they notice a factorization at each face level:
\begin{align}
\Sb_{F}  = {} & 1+\left(  t_{2}+t_{3}+t_{4}+\ldots\right)  f
\\[-2pt]
& +\left(  2t_{2}^{2}+5t_{2}t_{3}+3t_{3}^{2}+6t_{2}t_{4}+7t_{3}t_{4}
+4t_{4}^{2}+\cdots\right)  f^{2}  
\nonumber \\[-2pt]
& +(
5t_{2}^{3}+21t_{2}^{2}t_{3}+28t_{2}t_{3}^{2}+12t_{3}^{3}+28t_{2}^{2}
t_{4}+72t_{2}t_{3}t_{4} 
\nonumber \\[-2pt] & \qquad
{}+45t_{2}t_{4}^{2}+45t_{3}^{2}t_{4}+55t_{3}t_{4}^{2}+22t_{4}^{3}+\cdots
) f^{3}
\nonumber 
+\cdots
\\ 
\Sb_{F}  = {} & 1+\Sb_1  f + \Sb_1  \left(
2t_{2}+3t_{3}+4t_{4}+\ldots \right)f^2
 \\[-2pt]
& +\Sb_1  \left(  5t_{2}^{2}+16t_{2}t_{3}
+12t_{3}^{2}+23t_{2}t_{4}+33t_{3}t_{4}+22t_{4}^{2}+\cdots \right) f^{3}
+\cdots \nonumber
\end{align}
where $\Sb_1 \equiv t_2 + t_3 + t_4 + \ldots$.
That leads them to:
\begin{theorem}[The Geode factorization, \WR{}, 2025] \label{thm:geodefactorization}
There is a unique polyseries $\Gb $ satisfying $\Sb-1=\Sb_1\Gb.$
\end{theorem}
We review their inductive proof.  
Let $\Sm_n$ be the multiset of subdigons with $n$ faces and $\Sb_n=\Psi(\Sm_n)$.
For $n=0$ we have $\Sm_0 = [ \, |  \, ]$, 
and $\Sb_0=\Psi(\Sm_0)=1$. 

For $n \ge 1$, subdigons in $\Sm_n$ are necessarily formed by applying $\tri_k$ to $k$-tuples of subdigons with fewer than $n$ faces. 
$\tri_k$ adds a face, so for $n=1$ face, the subdigons all come from
$\tri_k(\Sm_0, \Sm_0, \ldots)$, each of which gives a singleton whose element is a subdigon with exactly one face, the unsubdivided $(k+1)$-gon. Those subdigons form $\Sm_1$, so $ \Sb_1 = \Psi(\Sm_1)= t_2 + t_3 + t_4 + \ldots$, agreeing with above.

For $n=2$ faces, $\Sm_2$, in each $\tri_k$ operation $\Sm_1$ must appear once and $\Sm_0$ must appear $k-1$ times. 
Applying $\Psi$  to a term of $\Sm_2$ gives: 
\begin{equation}
\Psi(\tri_k(\Sm_{j_1}, \Sm_{j_2}, \ldots, \Sm_{j_k}) ) = t_k \Psi(\Sm_1) (\Psi(\Sm_0))^{k-1} = t_k \Sb_1.
\end{equation}
$\Sb_1$ is a factor of each term of $\Sb_2=\Psi(\Sm_2)$; we conclude that $\Sb_1$ is a factor of $\Sb_2$.
Similarly each term $\tri_k(\Sm_{j_1}, \Sm_{j_2}, \ldots, \Sm_{j_k})$ of $\Sb_n$ for $n \ge 3$ has at least one non-zero $j_i$, so each term has at least one of $\Sb_1$, $\Sb_2$, ... $\Sb_{n-1}$ as a factor, so $\Sb_1$ as a factor.  
For $n \ge 1$ we conclude $\Sb_n$ has $\Sb_1$ as a factor,
so $\Sb_1$ is a factor of $\sum_{n\ge 1} \Sb_n = \Sb - 1$.
\hfill \qedsymbol{}
\section{The Geode Recurrence}
Theorem \ref{thm:geodefactorization} says $$
-1 + \Sb[ t_2,  t_3, \ldots] = (t_2 + t_3 + t_4 + \ldots) \Gb[t_2,  t_3, \ldots]
$$
For a non-degenerate type vector $\m  \ne [ \, ]$ (at least one non-zero component),
\begin{align}
[\T^{\m}] \Sb[ t_2,  t_3, \ldots] &= [\T^{\m}] \sum_{j \ge 2}  t_j \Gb[t_2,  t_3, \ldots]
=  \sum_{ \substack{ j \ge 2 \\ m_j > 0 }}  [\T^{\m - \vec{j}}] \Gb[t_2,  t_3, \ldots]
\\
\mC &= \sum_{ \substack{ j \ge 2 \\ m_j > 0 }}  G_{\m - \vec{j}} 
\end{align}
where analogously to the hyper-Catalans we denote the elements of the Geode array as $G_{\n} \equiv [\T^{\n}] \Gb$.
We define $\mD$ as the  number of \textbf{distinct shapes} of a subdigon of type $\m$,
and we define $L(\m)$, 
the \textbf{lessers} of a type $\m$, as the set of types $\k$ such that $\m=\k + \vec{j}$ for some natural $j \ge 2$.
Clearly  $|L(\m)|=\mD$ as we get a lesser for each non-zero component of $\m$.
\begin{equation}
\mD \equiv \sum_{\substack{j\ge 2 \\ m_j>0}} 1, \qquad
L(\m) \equiv \{ \k : \m=\k+\vec{j} \textrm{ for some natural } j \ge 2 \ \} 
.    
\end{equation}
\begin{theorem}[Lesser Geode Sum]\label{thm:geodelessersum}  
For $\m \ne [\,]$,
$\displaystyle \ 
\mC = \sum_{\k \in L(\m)} G_{\k} 
$.
\end{theorem}
To enumerate $L(\m)$ we simply subtract one from each non-zero component of $\m$.  

\renewcommand{\l}{{\mathbf l}}

Note that $\m-\vec{j}$ also appears in the hyper-Catalan Recurrence (Theorem \ref{thm:hcat}), which could be recast in terms of lessers, though we'd need more notation to replace some explicit appearances of $j$. 

Consider types with one distinct shape, $D_{\m}=1$.
Those subdigons are divided into all triangles, or all pentagons, or any other single kind of $k$-gon.
This question of division into $m$ $k$-gons was first considered by Fuss in 1795, and we'll refer to the two dimensional array of counts as the Fuss numbers~\cite{Aval2008}. 

We introduce vector expressions in parentheses as array indices, e.g. $C( m \vec{k})$ for $C[0,0, \ldots,0, m]$,
where there are $k-2$ zeros,
is the Fuss number giving the number of ways to divide a roofed polygon into $m$ $k+1$-gons.
Such a subdigon
has type $\m = m \vec{k}$ and $D_{\m}=1$ (provided  $\m \ne [\,]$, i.e. $m > 0$). 
From the Lesser Geode Sum Theorem, we get:

\begin{theorem}[Single Shape Geode Elements are hyper-Catalans] \label{thm:singleshapegeode}
$$G( m \vec{k}) = C( (m+1) \vec{k}) $$
\end{theorem}

We've confirmed that these particular Geode elements are Fuss numbers, which happens to be one of the conjectures of \WR{}

For a recurrence that we can more generally use to calculate Geode elements from hyper-Catalans, we imagine a function $j=X(\m)$ that returns the index $j$ letting us express $G_{\m}$ in terms of $C_{\m+\vec{j}}$ and the other $G_{\k}$ where $\k \ne \m \textrm{ and } \k \in L(\m+\vec{j})$.
\begin{theorem}[The Geode Recurrence]\label{thm:geoderecurrence}
$$
G_{\m} = C_{\m + \vec{X}(\m)} \quad  - \sum_{\k \in L(\m + \vec{X}(\m)) - \{ \m \}} G_{\k}
$$
\end{theorem}
$X(\m)$ can be any function that returns a valid index; as an example, let's define $X_k(\m)=k$ for $k \ge 2$ (our type vectors are indexed starting at $2$).

We get a different decomposition of $G_{\m}$ into an integer combination of $C_{\n}$ depending on which $k$ we choose.
Let's look at $G_{111}$ where we choose $X_2(\m)$ for $X(\m).$
The recurrence says:
\begin{equation}
    G_{111} = C_{211}-  G_{201}- G_{21} 
\end{equation}
By recursively expanding (using the $X_2$ recurrence) the $G$s that appear,
we can express $G_{111}$ in terms of hyper-Catalans.
\begin{align}
G_{201} &=C _{301}- G_3 = C_{301}- C_4 
\\[-2pt]
G_{21} & =C_{31}-G_3= C_{31}- C_4
\\[-2pt]
G_{111} &= C_{211}-  C_{301}+  C_4 - C_{31} + C_4 
\end{align}
The choice of $k$ determines the expansion.
We expand for $k=2$ through $6$:
\begin{align}
G_{111} = {} &  319  {}  \nonumber
\\[-2pt] 
= {} & C_{211}  - C_{301}  - C_{31} + 2 C_{4}
\\[-2pt] = {} & C_{121}  - C_{031}  - C_{13} +  2 C_{04}
\\[-2pt] = {} & C_{112}  - C_{013} - C_{103}+  2 C_{004}
\\[-2pt] = {} & C_{0301}+  C_{031}+  C_{1111}+  C_{301}+  C_{31}- C_{0211} - C_{2011}- C_{2101}  
\\[-2pt] & \qquad
- 4 C_{4}- 2 C_{04}+ 2 C_{3001}  \nonumber
\\[-2pt]  = {} & C_{03001}+  C_{031}+ C_{11101}+ C_{301}+ C_{31}  - C_{02101} - C_{20101} - C_{21001}
\\[-2pt] & \qquad
- 4 C_{4}- 2 C_{04}+ 2 C_{30001} \nonumber
\end{align}
The multitude of expressions for each Geode element seems to suggest that the Geode eliminates some redundancy inherent in the hyper-Catalans.
From the Geode factorization we know Geode elements are always natural numbers.
It is unknown what if anything is counted by the Geode array.

An $X(\m)$ that seems to lead to short expansions is as follows:
For a type $\m$, let $j=X(\m)$ be the index such that for all $i$, $m_j \ge m_i$ and also that if $m_j=m_i$ then $j \le i$.  In English, $j$ is the index of the biggest $m_j$, and is the smallest such $j$ in the event of a tie.

As we increase one index, we decrease all the others, one at a time.  Eventually an index is reduced to zero, which reduces the number of indices. Eventually we get to Geode elements with a single index, which are Fuss numbers.
So it's reasonably clear that, for the above described $X(\m)$, the Geode Recurrence recursively applied always terminates, always expresses a Geode element as a finite sum of hyper-Catalans (possibly with repetitions).

\section{Proving Wildberger's Geode Conjectures}

\WR{} conjecture a closed form for $G[m_2,m_3]$ (the Geode Bi-Tri array) and then generalize that into a closed form for $G[0,\ldots,0,m_k, m_{k+1}]$ where there are $k-2$ zeros.
These are special cases; not only are we restricted to $D_{\m}=2$ distinct shapes, the shapes must be \textbf{consecutive}, meaning $k+1$-gons and $k+2$-gons.
We prove both conjectures; the Bi-Tri case is a nice starting point.  We've so far avoided the closed form of the hyper-Catalan array, but now we need a bit of it for our proof. 

To prove Wildberger's simpler conjecture, we only need the Bi-Tri Array, the hyper-Catalans counting divisions into $m_2$ triangles, $m_3$ quadrilaterals and no other shapes.
From Theorem \ref{thm:hypercatalanclosedform}, it has the closed form: 
\begin{equation}
C[m_2,m_3] = \dfrac{( 2m_2 + 3m_3)!}{(1 + m_2+ 2m_3)! \,  m_2! \, m_3! } 
\end{equation}

\begin{theorem}[Geode Bi-Tri Closed Form] \label{thm:geodebitriclosedform}
Let
$$
H(m,n) =  \frac{(2m + 3n + 3)!}{(2m +  2n + 3) (m + n+ 1)  \ (m + 2n + 2)! \, m! \, n!}    .
$$
Then $G[m,n]=H(m,n).$
\end{theorem}
\begin{proof}
We proceed by induction on $n$.

\textbf{Base. } $n=0$.  Show $G[m,0]=H(m,0)$.
This is the case with type $\m=[m]$, $D_{\m}=1$, a single shape, namely triangles. That's $k=2$ in Theorem \ref{thm:singleshapegeode}, which tells us 
\begin{equation}
G[m,0]=G[m]=C[m+1]
\end{equation}
the $m+1$st Catalan number.
We compare that to:
\begin{equation}
H(m,0)
= \frac{(2m+ 3)!}{(2m+ 3)(m+ 1) (m+ 2)! \, m! }
= \frac{(2m+ 2)!}{(m+ 1)! (m+ 2)! }=C[m+1]
\end{equation}
and conclude $G[m,0]=H(m,0)$ for $m\ge 0$.

\textbf{Step. } We first show 
$H(m,n+1)=C[m+1,n+1]-H(m+1,n)$.
\begin{align}
& H(m,n+1) + H(m+1,n)  =
 \\ &
\dfrac{(2m+3n+5)!} {(2m \p 2n \p 5)(m \p n\p 2)  (m \p 2n \p3)! m! n!  } 
  \left( 
\dfrac{2m + 3n + 6}{(m \p 2n \p 4)(n\p1)}
\p\dfrac{1}{m\p1}
\right)  \nonumber
\\ 
&
= \dfrac{(2 m + 3 n + 5)!}{ (m + 2 n + 4)! \,  (m + 1)! \,  (n + 1)!}
= C[m+1, n+1] \quad\checkmark \nonumber
\end{align}
For $m \ge 0$, 
$G[m,n+1]=C[m+1,n+1]-G[m+1,n]$ (Theorem \ref{thm:geodelessersum}) and we just showed $H(m,n+1)=C[m+1,n+1]-H(m+1,n)$.
For the induction step we assume for all $m \ge 0$ that
$G[m,n]=H(m,n)$ and show that implies $G[m,n+1]=H(m,n+1)$.  
In particular we may assume $G[m+1,n]=H(m+1,n)$ so $G[m,n+1]=C[m+1,n+1]-G[m+1,n] =C[m+1,n+1]-H(m+1,n)=H(m,n+1)$ as required to conclude $G[m,n]=H(m,n)$. 
\end{proof}

Essentially the same argument applies to Wildberger's most general conjecture, of which the previous two are special cases.
It's a closed form expression for Geode elements with type vectors of the form $\m= m \overrightarrow{k} + n (\overrightarrow{k+1})$ for $k \ge 2$.
Again we'll need the closed form hyper-Catalans for these. 
Specializing Theorem \ref{thm:hypercatalanclosedform} to the two consecutive shape case, $m$ $k$-gons and $n$ $(k+1)$-gons,
\begin{align}
C( m \overrightarrow{k} + n (\overrightarrow{k+1}) ) &=  \dfrac{( km+(k+1)n )!}{(1 + (k-1)m+kn)! \,  m! \, n!} .
\end{align}

\begin{theorem}[Geode Elements with Two Consecutive Shapes] \label{thm:geodetwoconsecutive}
Let
\begin{align*}
& H(k,m,n) =
\frac {( k m +  (k + 1)(n +1))!}
{{ (k( m + n + 1) + 1)( m +  n + 1)} 
  ((k-1)m +  k(n + 1) )! \ m! \,n!}.
\end{align*}
Then
$G(m \overrightarrow{k} + n (\overrightarrow{k+1}) ) =H(k,m,n) $.
\end{theorem}
\begin{proof}
We proceed as before.

\textbf{Base. } $n=0$.  Show $ G(m 
\vec{k}) =H(k,m,0) $.  

The Single Shape Geode Theorem (Theorem \ref{thm:singleshapegeode}) tells us 
$G( m \vec{k}) = C( (m+1) \vec{k}) $
is a Fuss number, the number of subdigons subdivided into $m\! + \! 1 \ \ (k \!+\! 1)$-gons.
\begin{align}
H(k,m,0) 
&=\frac {( k (m+1) +1)!} {   (k( m + 1) + 1)( m + 1) \, ((k-1)m +  k )! \, m! }
\\
&=\frac {( k (m+1))!} {  ((k-1)(m+1) + 1 )! \, ( m + 1) ! }
=  C( (m+1) \vec{k})  = G( m \vec{k})  \nonumber
\quad \checkmark
\end{align}

\textbf{Step. }  We first show: 
\begin{align}
& H(k,m+1,n) + H(k,m,n+1)
\\
&
= 
\frac{( km+ kn + 2k + n + 1)!}
{ (k m +k n + 2k + 1) ( m +  n+ 2)  \, ( km +kn+2k-m-1 )! \, m! \, n! }
\nopagebreak \nonumber
\\ \nopagebreak &  \qquad \cdot  \left(
\frac{1}{m+1}
+
\frac{k m + kn + 2k + n+2}{(km+  k n + 2k  -m ) \, (n+1)}
\right)   \nonumber 
\\&  
=
\frac{( km+ kn + 2k + n + 1)!}
{ ( km +kn+2k-m )! \,( m+1)! \, (n+1) ! }
=C( (m+1) \overrightarrow{k} + (n+1) (\overrightarrow{k+1}) ) 
\nonumber 
\end{align}

The Lesser Geode Sum Theorem (Theorem \ref{thm:geodelessersum}) applied to $C( (m+1) \overrightarrow{k} + (n+1) (\overrightarrow{k+1}) ) $
says:
\begin{align}
& G( m \overrightarrow{k} + (n+1) (\overrightarrow{k+1}) )
+ G( (m+1) \overrightarrow{k} + n (\overrightarrow{k+1}) )
\\&
= C( (m+1) \overrightarrow{k} + (n+1) (\overrightarrow{k+1}) ) 
= H(k,m+1,n) + H(k,m,n+1) \nonumber
\end{align}

Now we assume
$G(m \overrightarrow{k} + n (\overrightarrow{k+1}) ) =H(k,m,n) $
for any $m \ge 0$ and show this implies
$G(m \overrightarrow{k} + (n+1) (\overrightarrow{k+1}) ) =H(k,m,n+1) $.
In particular we assume $G( (m+1) \overrightarrow{k} + n (\overrightarrow{k+1}) ) =H(k,m+1,n) $.
Those cancel in the above equality and we're left with:
\begin{align} &  H(k,m,n+1)
=G( m \overrightarrow{k} + (n+1) (\overrightarrow{k+1}) )
\end{align}
so we may conclude $G(m \overrightarrow{k} + n (\overrightarrow{k+1}) ) =H(k,m,n) $ for $m,n \ge 0$.
\end{proof}

\section{An unusual binomial coefficient identity}

The next thing we explore is the more general case with two distinct shapes, $D_{\m}=2$.
With just two distinct shapes, the Geode Recurrence (Theorem \ref{thm:geoderecurrence}) unrolls in a straightforward fashion. Let's use index function $X_j$, meaning we add one to $\vec{j}$'s scalar when iterating to get the hyper-Catalan index.
We have $j \ne k, m \ge 1, n \ge 1$.

\begin{align}
 G(m \vec{j}+ n \vec{k} ) &= C( (m+1) \vec{j}+ n \vec{k} ) -  G( (m+1) \vec{j}+ (n-1) \vec{k} )
\\ 
G( (m+1) \vec{j}+ (n-1) \vec{k} ) &= C( (m+2) \vec{j}+ (n-1) \vec{k} ) -  G( (m+2) \vec{j}+ (n-2) \vec{k} )
\nonumber
\\ 
G( (m+2) \vec{j}+ (n-2) \vec{k} ) &= C( (m+3) \vec{j}+ (n-2) \vec{k} ) -  G( (m+3) \vec{j}+ (n-3) \vec{k} )
\nonumber
\end{align}
The pattern continues until we encounter a zero scalar on $\vec{k}$.  We conclude:

\begin{theorem}[Geode elements with two distinct shapes] \label{thm:geodetwoshapes}
For $j \ne k, m \ge 1, n \ge 1$,
$$
G(m\vec{j} + n \vec{k}) = \sum_{i=0}^n (-1)^i \, C( \, (m+1+i)\vec{j} + (n-i)\vec{k} \, )
$$\vspace{-18pt}
\end{theorem}
This contains two different decompositions of a given $G(m\vec{j} + n \vec{k})$ because we may swap $m$ and $j$ with $n$ and $k$ respectively.

We don't hold out much hope for a closed form for the Geode like the one for the hyper-Catalans (Theorem \ref{thm:hypercatalanclosedform}) because in the simplest non-consecutive shape case, Bi-Quad ($j=2, k=4$), relatively large primes appear along the diagonal that would be difficult to explain with a formula like that of the hyper-Catalans:
\begin{align}
&G_{101}=C_{201}-C_3=28 - 5=23
\\
&
G_{202}=C_{302}-C_{401}+C_5= 2002 - 495 + 42 = 1549
\\
&
G_{303}=C_{403}-C_{502}+C_{601}-C_7 = 193800 -55692 +  8008 -429 = 145687
\end{align}

Let's explore the simplest case, the Geode Bi-Tri Array, by setting $j=2$ and $k=3$.
\begin{equation}
G[m,n] = \sum_{i=0}^n (-1)^i \, C[m+1+i, n-i]  
\end{equation}
Expanding with Theorem \ref{thm:hypercatalanclosedform}, namely
$ 
C[m,n] = \frac{( 2m + 3n)!}{(1 + m+ 2n)! \,  m! \, n! } 
$,
and equating $G[m,n]$ from the Geode Bi-Tri Closed Form (Theorem \ref{thm:geodebitriclosedform}), we get:
\begin{align}
 & 
  \frac{(2m + 3n+  3)!}{(2m +  2n +  3)(m +  n +  1) (m +  2n + 2)! \, m! \, n!}
\\ & \qquad =   
 \sum_{i=0}^{n} (-1)^i \,  \frac{(2+  2m + 3n  -i)!}{(2 +m+2n-i)! \,  (1+m+i)! \, (n-i)! }  \nonumber
\end{align}
Routine manipulation and the substitution $s=m+n$ yields:
\begin{theorem}[An Unusual Binomial Coefficient Identity]
\begin{align*}
& \binom{ s }{n}
=  (3+2s ) \sum_{i=0}^{n} (-1)^i \,  
 \dfrac{ 
 \binom{2+s+n}{i}  \binom{s+1}{n-i}   }
 { 
 (i+1)   \binom{  3+ 2s+n}{i+1} } 
\end{align*}
\end{theorem}
To generalize, 
we use Theorem \ref{thm:geodetwoshapes}, equating it to the result of Theorem \ref{thm:geodetwoconsecutive}, substituting for the hyper-Catalans using Theorem \ref{thm:hypercatalanclosedform}.
For $k \ge 2, m \ge 1, n \ge 1$, equating the two expressions for $G(m\vec{k} + n (\vec{k+1})) $ and again substituting $m \to n-s$ yields the general result. 
For presentation, let's juggle the variables $k \to t, n \to k, s \to n$:
\begin{theorem}[A Family of Binomial Coefficient Identities] 
\begin{align*}
\binom {n}{k}=
(t+1 + t n )\sum_{i=0}^k (-1)^i \,  
\dfrac{
\binom{t + (t-1) n + k }{i}  
\binom{n+1}{k-i}
}
{(i+1) 
\binom{ t+1 + t n + k }{i+1}}
\end{align*}
\end{theorem}
The above derivation only holds for $t \ge 2$, but sympy seems to indicate that $t=1$ works fine, as do integers $t \le 0$, which give rise to binomial coefficients with negative tops, easily interpreted with falling powers. $t=1$ gives:
\begin{align}
\binom {n}{k}=
(2 +  n )
\sum_{i=0}^k (-1)^i \,  
\dfrac{
\binom{k+1 }{i}  
\binom{n+1}{k-i}
}
{(i+1) 
\binom{ 2+ n + k }{i+1}}
\end{align}
Let's write out $\binom{4}{2}$: $n=4, k=2, t=1$.
\begin{align}
&
6 \left( \dfrac{  \binom{3 }{0}  \binom{5}{2} }{1 \binom{ 8 }{1}} - \dfrac{\binom{3}{1} \binom{5}{1} }{2 \binom{8}{2}} + \dfrac{  \binom{3}{2} \binom{5}{0} } {3 \binom{8}{3}} \right)
 = 6 \left( \frac{5}{4} - \frac{15}{56}  + \frac{1}{56} \right)
= 6 = \binom {4}{2} \quad\checkmark
\end{align}
Nothing particularly special about $t=1$; here's $t=100$:
\begin{align}
\binom{4}{2} 
= 501 \left( \textstyle \frac{ \binom{5}{2 } \binom{498}{0 }}{1 \binom{503}{1 }} - \frac{ \binom{5}{1 } \binom{498}{1 }}{2 \binom{503}{2 }} + \frac{ \binom{5}{0 } \binom{498}{2 }}{3 \binom{503}{3 }} \right)
\textstyle = 501 \left( \frac{10}{503} -\frac{1245}{126253}  + \frac{41251}{21084251} \right) 
\end{align}

Thanks to Marko Riedel, who deftly supplied two proofs 
when I posted this identity at math.stackexchange.com/questions/4906245 back in April, 2024.

\section{The Geode and the Encyclopedia}

\WR{} list around fifteen sequences derived from the hyper-Catalans that are found in the Online Encyclopedia of Integer Sequences~\cite{OEIS}.
Here we generate the analogous sequences for the Geode and find many that have yet to appear in the OEIS.

We know the single-shape Geode elements are Fuss numbers, which are well represented in the OEIS. 
We've found 
$G[n]=C[n+1]$ is A000108, the Catalans; 
$G[0,n]=C[0,n+1]$ is A001764, the Fuss numbers for decomposition into quadrilaterals;
$G[0,0,n]=C[0,0,n+1]$ is A002293, Fuss pentagons;
$G[0,0,0,n]=C[0,0,0,n+1]$ is A002294, Fuss hexagons.

Going beyond a single shape to two distinct shapes, the corresponding hyper-Catalan slices appear in OEIS but these Geode slices do not. 
\begin{align*}
G[n,1] &=
3, 16, 70, 288, 1155, 4576, 18018, 70720, 277134, 1085280  
, \cdots \\[-4pt]
G[1,n] &=
2, 16, 110, 728, 4760, 31008, 201894, 1315600, 8584290 
, \cdots \\[-4pt]
G[n,2] &=
12, 110, 702, 3850, 19448, 93366, 433160, 1961256, 8721000 
, \cdots \\[-4pt]
G[1,0,n] &=
2, 23, 224, 2091, 19250, 176410, 1615068, 14793944 
, \cdots \\[-4pt]
G[0,1,n] &=
3, 33, 315, 2907, 26565, 242190, 2208843, 20173560 
, \cdots \\[-4pt]
G[n,0,1] &=
4, 23, 106, 453, 1870, 7579, 30394, 121108, 480624, 1902470 
, \cdots \\[-4pt]
G[0,0,1,n] &=
4, 56, 684, 8096, 95004, 1113024, 13050492, 153282272 
, \cdots
\end{align*}

The Little Schroeder numbers, A001003, are a one dimensional projection of the hyper-Catalans, the coefficients of $v$ in 
$\Sb[v,v^2,v^3,\ldots]$.
The Little Schroeder Geode numbers come from $\Gb[v,v^2,v^3,\ldots]$:
$$
\cdots + 16514 v^{7} + 3376 v^{6} + 706 v^{5} + 152 v^{4} + 34 v^{3} + 8 v^{2} + 2 v + 1. $$

The Riordan numbers are the coefficients of $e$ in $\Sb[e^2,e^3,e^4,\ldots]$, A005043.
The Riordan Geode numbers come from $\Gb[e^2,e^3,e^4,\ldots]$:
$$
\cdots + 982 e^{9} + 371 e^{8} + 141 e^{7} + 55 e^{6} + 21 e^{5} + 9 e^{4} + 3 e^{3} + 2 e^{2} + 1
$$

The Cayley numbers (A033282 dovetailed) are a two dimensional projection of the hyper-Catalans, coefficients of $\Sb[vf, v^2f, v^3f, \ldots]$.  
The Cayley Geode numbers are coefficients of $\Gb[vf, v^2f, v^3f, \ldots]$:
\begin{align*}
&  
1 
+ v^1(2 f)  
+ v^{2} \left(5 f^{2} + 3 f\right)
+ v^{3} \left(14 f^{3} + 16 f^{2} + 4 f\right) 
\\[-4pt] &
+ v^{4} \left(42 f^{4} + 70 f^{3} + 35 f^{2} + 5 f\right) 
\\[-4pt] & 
+ v^{5} \left(132 f^{5} + 288 f^{4} + 216 f^{3} + 64 f^{2} + 6 f\right) 
\\[-4pt] & 
+ v^{6} \left(429 f^{6} + 1155 f^{5} + 1155 f^{4} + 525 f^{3} + 105 f^{2} + 7 f\right)
\\[-4pt] & 
+ v^{7} \left(1430 f^{7} + 4576 f^{6} + 5720 f^{5} + 3520 f^{4} + 1100 f^{3} + 160 f^{2} + 8 f\right) 
+ \cdots
\end{align*}

\section{Conclusion}

The Geode is a mysterious object that appears to underlie the hyper-Catalans.
By proving Wildberger's conjectures, we make some small progress toward understanding it.

Wildberger's conjectures were restricted to the case of at most two distinct, consecutive shapes.
Glancing beyond, the Geode Bi-Quad array, two distinct non-consecutive shapes, has large primes, indicating its formula is not the simply the ratio of factorials and other factors of small linear combinations of $m_2$ and $m_4$.
For three or more distinct shapes, the straightforward unrolling of Theorem \ref{thm:geodetwoshapes} becomes a more complicated tree-like descent.

The Geode is something new.
Thanks to Dan Eilers of Irvine, CA for finding A243660 (the x = 1 + q Narayana
triangle at m = 2) is essentially the Geode Bi-Tri array $\Gb[m_2,m_3]$. Similarly, A243661 (same at
m = 3) is essentially the Geode Tri-Quad array $\Gb[0, m_3, m_4]$.
The various sequences derived from the Geode beyond these and the Fuss numbers do not seem to appear in the OEIS.
We still have no idea of its general formula (beyond the Geode Recurrence) or what if anything it counts.

The connection between the algebra of multisets of subdivided polygons and the zeros of polynomials
is a key that unlocks new results in combinatorics.
Here we were able to find combinatorial identities without Lagrange Inversion, differentiation or other advanced combinatorial techniques, employing only the geometric polynomial formula, the Geode factorization and the multinomial theorem.
We look forward to pursuing more identities in the future.

\section{Developments since publication of Wildberger and Rubine}

The above was mostly written in 2024 and early 2025.
Since the publication of Wildberger and Rubine, online in April, 2025, we've received proofs of the Geode Conjectures.

Ira Gessel (personal communication) sent a proof of Theorem $\ref{thm:geodetwoconsecutive}$ that independently shows all three Wildberger conjectures discussed above. 
Upon review, Professor Gessel commented that his proof was essentially the same as the one given here.

Gessel also proves an additional conjecture from Wildberger and Rubine, and uncovers a new array.

\begin{theorem}
With $2k$ parameters, $\Gb[-f,f,\ldots,-f,f] = \sum_{n \ge 0} k^n f^n$.
\end{theorem}
\begin{proof}
We (says Gessel in our recounting) begin from $\Sb_1=t_2+t_3+t_4+\ldots$ and $ \Sb_1 \Gb = \Sb -1 = \sum_{n \ge 2} t_n \Sb^n$. 

Pulling out $\Sb_1$, factoring $\Sb^n-1$, then recognizing and solving for $\Gb$, we have:
\begin{align}
\Sb_1 \Gb &= \Sb_1 + \sum_{n \ge 2} t_n(\Sb^n -1)
\\    \Gb &=  1 + \sum_{n \ge 2} t_n \dfrac{ \Sb^n -1} { \Sb_1}
\\  \Gb&= 1 + \sum_{n \ge 2}  t_n \dfrac{ \Sb -1} { \Sb_1}( 1 + \Sb + \Sb^2 + \ldots + \Sb^{n-1})
\\   \Gb&= 1 + \Gb \sum_{n \ge 2}  t_n ( 1 + \Sb + \Sb^2 + \ldots + \Sb^{n-1})
\\   \Gb&=\left( 1 -  \sum_{n \ge 2}  t_n ( 1 + \Sb + \Sb^2 + \ldots + \Sb^{n-1})
\right)^{-1}
\end{align}
Now we evaluate the various terms at $-f, f,$ etc.
Clearly $\Sb_1[-f, f, \ldots, -f, f ]=-f+f+\ldots -f+f=0$ so per  $\Sb_1 \Gb = \Sb -1$ we conclude $\Sb[-f,f, \ldots]=1.$
We see the sum of $n$ powers of $\Sb[-f,f,\ldots]$ is simply $n$, and $t_n=(-1)^{n-1} f$ for $2 \le n \le 2k+1$, else $0$,
\begin{align}
 \Gb[-f, f, \ldots]&=\left( 1 - f \sum_{n=2}^{2k+1}  (-1)^{n-1} \, n \right)^{-1}
\\& = \left(1 - (-2 + 3 - 4 + 5 \ldots - 2k + (2k+1) ) \, f \right)^{-1}  
\\ \Gb[-f, f, \ldots]& = \dfrac{ 1}{ 1 - kf } 
\end{align}
which is of course $\sum_{n\ge 0} k^n f^n$.
\end{proof}
Gessel further observes for any parameters $t_2, t_3, \ldots$ where $t_2+t_3+ \ldots = 0$, then (formally):
\begin{align}
    G[t_2, t_3, \ldots] = \left( 1- \sum_{n \ge 2} n t_n\right) ^{-1} 
\end{align}

In a later email, Gessel shows results still hold when a general $t_1$ is included, $\Sb=1+ \sum_{n \ge 1} t_n \Sb ^n$, and in that context he goes on to define $\Ub$ by $\Sb=(1-\Ub)^{-1}.$ Then $\Ub=1-\Sb^{-1}.$ Dividing $\Sb$'s defining equation by $\Sb$ gives $1=\Sb^{-1}+\sum_{\n\ge 1} t_n \Sb^{n-1}$ so $\Ub= \sum_{\n\ge 1} t_n \Sb^{n-1}$, so necessarily has natural number coefficients.

He goes on to define $\Hb$ by $\Ub=\Sb_1 \Hb$.  We have $\Ub \Sb  = \Sb -1=\Sb_1 \Gb $ so $\Hb = \Ub/\Sb_1 = \Gb / \Sb$.  
Substituting $\Sb \Hb$ for $\Gb$ and dividing by $\Sb$, noting $\Sb^{-1}=1-\Ub = 1 - \Hb \Sb_1$, we have 
\begin{align}
    \Hb& = 1- \Hb\Sb_1 + \Hb \sum_{n \ge 1} t_\n ( 1 + \Sb + \Sb^2 + \ldots + \Sb^{n-1})
\\ \Hb &= 1 + \Hb \sum_{n \ge 1} t_n(\Sb+\Sb^2 + \ldots +\Sb^{n-1} ) \label{eqn:ges1}
\\  \Hb^{-1} &= 1-  \sum_{n \ge 1} t_n(\Sb+\Sb^2 + \ldots +\Sb^{n-1} )
\end{align}

He concludes from equation (\ref{eqn:ges1}) that $\Hb = \Gb/ \Sb$ has natural number coefficients. 
Clearly $\Hb^{-1}$ has integer coefficients, negative except for the constant. 
From 
$\Sb^{-1} = 1 - \Ub = 1 -  \Sb_1 \Hb$ we have:
 \begin{equation}
     \Sb = \left( 1 - \dfrac{\Sb_1}{ 1-  \sum_{n \ge 1} t_n(\Sb+\Sb^2 + \ldots +\Sb^{n-1} )} \right)^{-1}
 \end{equation}

Let's write the vertex layering of $\Hb$, made with software that assumes the coefficients start at $t_2$ and the types start at $m_2$, i.e. $\Sb= 1+ \sum_{n \ge 2}t_n \Sb^n$ and $\Sb_1 = t_2 +t_3 + \ldots$.

\begin{align}
\Hb[vt_2,& v^2t_3, v^3 t_4, \ldots] = 1 +  v t_{2}  +
v^{2} (2 t_{2}^{2} + 2 t_{3}) + 
v^{3} (5 t_{2}^{3} + 8 t_{2} t_{3} + 3 t_{4}) +
{} \nonumber \\& v^{4} (14 t_{2}^{4} + 30 t_{2}^{2} t_{3} + 13 t_{2} t_{4} + 7 t_{3}^{2} + 4 t_{5}) + 
{} \nonumber \\& v^{5} (42 t_{2}^{5} + 112 t_{2}^{3} t_{3} + 51 t_{2}^{2} t_{4} + 54 t_{2} t_{3}^{2} + 19 t_{2} t_{5} + 21 t_{3} t_{4} + 5 t_{6}) + 
{} \nonumber \\& v^{6} (132 t_{2}^{6} + 420 t_{2}^{4} t_{3} + 196 t_{2}^{3} t_{4} + 308 t_{2}^{2} t_{3}^{2} + 79 t_{2}^{2} t_{5} + 171 t_{2} t_{3} t_{4}  \nonumber \\[-5pt]&\qquad + 26 t_{2} t_{6} + 30 t_{3}^{3} + 29 t_{3} t_{5} + 15 t_{4}^{2} + 6 t_{7}) + 
{} \nonumber \\& v^{7} (429 t_{2}^{7} + 1584 t_{2}^{5} t_{3} + 750 t_{2}^{4} t_{4} + 1560 t_{2}^{3} t_{3}^{2} + 316 t_{2}^{3} t_{5}  \nonumber\\[-5pt]&\quad + 1012 t_{2}^{2} t_{3} t_{4} + 115 t_{2}^{2} t_{6} + 352 t_{2} t_{3}^{3} + 251 t_{2} t_{3} t_{5}  \nonumber\\[-5pt]&\quad  + 129 t_{2} t_{4}^{2} + 34 t_{2} t_{7} + 135 t_{3}^{2} t_{4} + 38 t_{3} t_{6} + 40 t_{4} t_{5} + 7 t_{8}) + 
{} \ldots
\end{align}
We can see the Catalans going down, and also a $1t_2+2t_3+3t_4+\ldots$ but beyond that 
a quick search didn't turn up anything in the OEIS.

The four Geode conjectures were also proven by Tewodros Amdeberhan and Doron Zeilberger, in an arXiv paper uploaded June 22, 2025 \cite{Amdeberhan2025}.
\vspace{10pt}
\noindent
%
\hrule


\begin{bibdiv}[]
\begin{biblist} 

\bib{Amdeberhan2025}{article}{
author = {Tewodros Amdeberhan},
author = {Doron Zeilberger},
title={\href{https://www.arxiv.org/pdf/2506.17862}{Proofs of Three Geode Conjectures}},
journal = {ArXiv.org},
year = {2025},
url = {https://www.arxiv.org/pdf/2506.17862},
note = {https://www.arxiv.org/pdf/2506.17862}
}

\bib{Aval2008}{article}{
author = {Nicolas Fuss},
title={\href{https://www.math.ucla.edu/~pak/lectures/Cat/Fuss1.pdf}{Solutio quaestionis, quot modis polygonum n laterum
in polygona m laterum, per diagonales resolui queat [Solution to the Question: In how many ways can a polygon with n sides be resolved by its diagonals?]}},
journal = {Nova Acta Acad. Sci. Imp. Pet.},
volume = {9},
pages = {243-251},
year = {1795},
}



\bib{Erdelyi1940}{article}{
title={\href{https://www.cambridge.org/core/services/aop-cambridge-core/content/view/4216D1DDB6DF23ADF56FBF3A8C7A6A7B/S0950184300002640a.pdf/some-problems-of-non-associative-combinations-2.pdf}{Some problems of non-associative combinations (2)}},
volume={32}, 
journal={Edinburgh Mathematical Notes}, 
publisher={Cambridge University Press}, 
author={Erd\'elyi, A.},
author={Etherington, I. M. H.},
year={1940},
pages={vii-xii},
}

\bib{French2018}{misc}{
title={\href{https://dc.etsu.edu/etd/3392}{Vector Partitions}},
author={French, Jennifer},
publisher={East Tennessee State University},
year={2018},
note={Electronic Theses and Dissertations. Paper 3392.},
url={https://dc.etsu.edu/etd/3392}
}

\bib{Graham1989}{book}{
  address = {Reading},
  author = {Graham, Ronald L. and Knuth, Donald E. and Patashnik, Oren},
  publisher = {Addison-Wesley},
  title = {Concrete mathematics: a foundation for computer science},
 year = {1989}
}


\bib{Knuth1997}{book}{
author = {Knuth, Donald E.},
title = {The Art of Computer Programming, Volume 1 (3rd Ed.): Fundamental Algorithms},
year = {1997},
isbn = {0201896834},
publisher = {Addison Wesley Longman Publishing Co., Inc.},
address = {USA}
}



\bib{Mane2024}{misc}{
author={Mane, S. R.},
year={2024},
title={\href{https://arxiv.org/pdf/1607.04144}{Multiparameter Fuss–Catalan numbers with application to algebraic equations}},
publisher={arXiv.org},
url={https://arxiv.org/pdf/1607.04144}
}


\bib{OEIS}{misc}{
author={OEIS Foundation Inc.},
year={2024},
title={\href{https://oeis.org/A000108}{The On-Line Encyclopedia of Integer Sequences}},
url={http://oeis.org}
}

\bib{Raney1960}{article}{
  title={\href{https://www.ams.org/journals/tran/1960-094-03/S0002-9947-1960-0114765-9/S0002-9947-1960-0114765-9.pdf}{Functional composition patterns and power series reversion}},
  author={Raney, George N.},
  journal={Trans am math soc},
  year={1960},
  volume={94},
  pages={441-451}
}


\bib{Wildberger2025}{article}{
author={Wildberger, N. J.},
author={Rubine, Dean},
title = {\href{https://doi.org/10.1080/00029890.2025.2460966}{A Hyper-Catalan Series Solution to Polynomial Equations, and the Geode}},
journal = {The American Mathematical Monthly},
volume = {132},
number = {5},
pages = {383--402},
year = {2025},
publisher = {Taylor \& Francis},
}



\end{biblist}
\end{bibdiv}
\end{document}